\documentclass[12pt, reqno]{amsart}
\usepackage{amsmath, amsthm, amscd, amsfonts, amssymb, graphicx, color}
\usepackage[bookmarksnumbered, colorlinks, plainpages]{hyperref}

\newtheorem{theorem}{Theorem}[section]
\newtheorem{lemma}[theorem]{Lemma}

\newtheorem{corollary}[theorem]{Corollary}
\theoremstyle{definition}

\theoremstyle{remark}

\numberwithin{equation}{section}

\newcommand{\cA}{\mathcal{A}}
\newcommand{\cB}{\mathcal{B}}
\newcommand{\cC}{\mathcal{C}}
\newcommand{\cE}{\mathcal{E}}
\newcommand{\cF}{\mathcal{F}}
\newcommand{\cK}{\mathcal{K}}
\newcommand{\cM}{\mathcal{M}}

\newcommand{\mC}{\mathbb{C}}
\newcommand{\mN}{\mathbb{N}}
\newcommand{\mR}{\mathbb{R}}
\newcommand{\mZ}{\mathbb{Z}}

\newcommand{\fa}{\mathfrak{a}}
\newcommand{\fb}{\mathfrak{b}}
\newcommand{\fc}{\mathfrak{c}}

\newcommand{\ff}{\mathfrak{f}}
\newcommand{\fg}{\mathfrak{g}}
\newcommand{\fr}{\mathfrak{r}}
\newcommand{\fs}{\mathfrak{s}}

\newcommand{\fA}{\mathfrak{A}}

\newcommand{\Ind}{\operatorname{Ind}}
\newcommand{\Co}{\operatorname{Co}}
\newcommand{\Op}{\operatorname{Op}}

\begin{document}
\setcounter{page}{1}

\title[Fredholmness and Index of Simplest Weighted SIOS]
{Fredholmness and Index of Simplest\\
Weighted Singular Integral Operators\\
with Two Slowly Oscillating Shifts}

\author[A. Karlovich]{Alexei Yu. Karlovich}
\address{
Centro de Matem\'atica e Aplica\c{c}\~oes (CMA) and
Departamento de Matem\'atica,
Faculdade de Ci\^encias e Tecnologia,
Universidade Nova de Lisboa,
Quinta da Torre,
2829--516 Caparica,
Portugal}
\email{\textcolor[rgb]{0.00,0.00,0.84}{oyk@fct.unl.pt}}

\dedicatory{Dedicated to Professor Yuri I. Karlovich on his 65th anniversary}




\begin{abstract}
Let $\alpha$ and $\beta$ be orientation-preserving diffeomorphisms
(shifts) of $\mR_+=(0,\infty)$ onto itself with the only fixed points $0$ and
$\infty$, where the derivatives $\alpha'$ and $\beta'$ may have discontinuities
of slowly oscillating type at $0$ and $\infty$. For $p\in(1,\infty)$, we consider
the weighted shift operators $U_\alpha$ and $U_\beta$ given on the Lebesgue space
$L^p(\mR_+)$ by $U_\alpha f=(\alpha')^{1/p}(f\circ\alpha)$ and $U_\beta f=
(\beta')^{1/p}(f\circ\beta)$. For $i,j\in\mZ$ we study the simplest weighted singular
integral operators with two shifts $A_{ij}=U_\alpha^i P_\gamma^++U_\beta^j P_\gamma^-$
on $L^p(\mR_+)$, where $P_\gamma^\pm=(I\pm S_\gamma)/2$ are operators
associated to the weighted Cauchy singular integral operator
$$
(S_\gamma f)(t)=\frac{1}{\pi i}\int_{\mR_+}
\left(\frac{t}{\tau}\right)^\gamma\frac{f(\tau)}{\tau-t}d\tau
$$
with $\gamma\in\mC$ satisfying $0<1/p+\Re\gamma<1$. We prove that
the operator $A_{ij}$ is a Fredholm operator on $L^p(\mR_+)$ and
has zero index if
\[
0<\frac{1}{p}+\Re\gamma+\frac{1}{2\pi}\inf_{t\in\mR_+}(\omega_{ij}(t)\Im\gamma),
\quad
\frac{1}{p}+\Re\gamma+\frac{1}{2\pi}\sup_{t\in\mR_+}(\omega_{ij}(t)\Im\gamma)<1,
\]
where $\omega_{ij}(t)=\log[\alpha_i(\beta_{-j}(t))/t]$ and $\alpha_i$, $\beta_{-j}$
are iterations of $\alpha$, $\beta$. This statement extends
an earlier result obtained by the author, Yuri Karlovich, and Amarino Lebre for $\gamma=0$.
\end{abstract}
\maketitle
\section{Introduction}
Let $\cB(X)$ be the Banach algebra of all bounded linear operators
acting on a Banach space $X$ and let $\cK(X)$ be the ideal of all
compact operators in $\cB(X)$. An operator $A\in\cB(X)$ is called
\textit{Fredholm} if its image is closed and the spaces $\ker A$
and $\ker A^*$ are finite-dimensional. In that case the number
\[
\Ind A:=\dim\ker A-\dim\ker A^*
\]
is referred to as the {\it index} of $A$ (see, e.g., \cite[Chap.~4]{GK92}).
For $A,B\in\cB(X)$, we will write $A\simeq B$ if $A-B\in\cK(X)$.
Recall that an operator $B_r\in\cB(X)$
(resp. $B_l\in\cB(X)$) is said to be a right (resp. left) regularizer for $A$ if
\[
AB_r\simeq I \quad(\mbox{resp.}\quad B_lA\simeq I).
\]
It is well known that an operator $A$ is Fredholm on $X$ if and only if it
admits simultaneously a right and a left regularizers. Moreover, each right
regularizer differs from each left regularizer by a compact operator
(see, e.g., \cite[Chap.~4, Section 7]{GK92}).

Following Sarason \cite[p.~820]{S77}, a bounded continuous function $f$ on $\mR_+=(0,\infty)$ is called
slowly oscillating (at $0$ and $\infty$) if for each (equivalently,
for some) $\lambda\in(0,1)$,
\[
\lim_{r\to s}\sup_{t,\tau\in[\lambda r,r]}|f(t)-f(\tau)|=0
\quad\mbox{for}\quad
s\in\{0,\infty\}.
\]
The set $SO(\mR_+)$ of all slowly oscillating functions forms a
$C^*$-algebra. This algebra properly contains $C(\overline{\mR}_+)$,
the $C^*$-algebra of all continuous functions on $\overline{\mR}_+
:=[0,+\infty]$.
Note that this notion of slow oscillation does not involve
any differentiability. Various modifications of it were studied in many
works, just to mention a few, see \cite{KS11,KLH12,KLH13,MSS14,P80,S07}.
On the other hand, there is also another (stronger) notion of slow oscillation (or slow variation)
in the literature. That notion is defined by using at least the first derivative
of a function and goes back to Grushin \cite{G70}, it was extensively used
since then in works on pseudodifferential operators and their applications,
see e.g., \cite{K82,R92,R98,RRS04}.

Suppose $\alpha$ is an orientation-preserving
diffeomorphism of $\mR_+$ onto itself, which has only two fixed
points $0$ and $\infty$. We say that $\alpha$ is a slowly oscillating
shift if $\log\alpha'$ is bounded and $\alpha'\in SO(\mR_+)$. The
set of all slowly oscillating shifts is denoted by $SOS(\mR_+)$.

We suppose that $1<p<\infty$. It is easily
seen that if $\alpha\in SOS(\mR_+)$, then the shift operator
$W_\alpha$ defined by $W_\alpha f=f\circ\alpha$ is bounded and
invertible on all spaces $L^p(\mR_+)$ and its inverse is given
by $W_\alpha^{-1}=W_{\alpha_{-1}}$, where $\alpha_{-1}$ is the
inverse function to $\alpha$. Along with $W_\alpha$ we consider
the weighted shift operator
\[
U_\alpha:=(\alpha')^{1/p}W_\alpha
\]
being an isometric isomorphism of the Lebesgue space $L^p(\mR_+)$
onto itself. It is clear that $U_\alpha^{-1}=U_{\alpha_{-1}}$.

Let $\Re\gamma$ and $\Im\gamma$ denote the real and imaginary part
of $\gamma\in\mC$, respectively. As usual, $\overline{\gamma}=\Re\gamma-i\Im\gamma$
denotes the complex conjugate of $\gamma$. If $\gamma\in\mC$ satisfies
$0<1/p+\Re\gamma<1$, then the weighted Cauchy singular integral
operator $S_\gamma$, given by
\[
(S_\gamma f)(t):=\frac{1}{\pi i}\int_{\mR_+}
\left(\frac{t}{\tau}\right)^\gamma\frac{f(\tau)}{\tau-t}d\tau,
\]
where the integral is understood in the principal value sense,
is bounded on the Lebesgue space $L^p(\mR_+)$ (see, e.g.,
\cite[Section~1.10.2]{DS08},
\cite{D79},
\cite[Section~2.1.2]{HRS94},
\cite[Proposition~4.2.11]{RSS11}). Put
\[
P_\gamma^\pm:=(I\pm S_\gamma)/2.
\]
Consider the weighted singular integral operators with
two shifts $\alpha,\beta\in SOS(\mR_+)$ given by
\begin{equation}\label{eq:operator}
A_{ij}:=U_\alpha^iP_\gamma^++U_\beta^jP_\gamma^-,
\quad i,j\in\mZ.
\end{equation}
Here, by definition, $U_\alpha^i=(U_\alpha^{-1})^{|i|}$ and $U_\beta^{j}=(U_\beta^{-1})^{|j|}$
for negative $i$ and $j$, respectively.

In \cite{KKL14} we called such operators (with $\gamma=0$) simplest singular integral
operators with two shifts because they do not involve functional coefficients, in contrast
to as it is traditionally considered (see, e.g., \cite{KKL11a,KKL11b,KL94,L77}). In that
paper we proved that if $\gamma=0$, then all operators \eqref{eq:operator} are Fredholm
and their indices are equal to zero. This paper is a sequel of \cite{KKL14}, here our aim
is to extend that result to the case of $\gamma\in\mC$ satisfying $0<1/p+\Re\gamma<1$.

For a shift $\alpha\in SOS(\mR_+)$, put $\alpha_0(t):=t$ and $\alpha_i(t):=\alpha[\alpha_{i-1}(t)]$
for every $i\in\mZ$ and $t\in\mR_+$.
\begin{theorem}[Main result]
\label{th:main}
Suppose $1<p<\infty$ and $\gamma\in\mC$ is such that $0<1/p+\Re\gamma<1$. If
$\alpha,\beta\in SOS(\mR_+)$, $i,j\in\mZ$, and
\[
0<\frac{1}{p}+\Re\gamma+\frac{1}{2\pi}\inf_{t\in\mR_+}(\omega_{ij}(t)\Im\gamma),
\quad
\frac{1}{p}+\Re\gamma+\frac{1}{2\pi}\sup_{t\in\mR_+}(\omega_{ij}(t)\Im\gamma)<1,
\]
where $\omega_{ij}(t)=\log[\alpha_i(\beta_{-j}(t))/t]$, then the operator $A_{ij}$ given by
\eqref{eq:operator} is Fredholm on the space $L^p(\mR_+)$ and its index is equal to zero.
\end{theorem}
By \cite[Corollary~2.5]{KKL14}, if $\alpha,\beta\in SOS(\mR_+)$,
then $\delta_{ij}:=\alpha_i\circ\beta_{-j}$ belongs to $SOS(\mR_+)$ for all $i,j\in\mZ$. Since
$A_{ij}=U_\beta^j(U_{\delta_{ij}}P_\gamma^++P_\gamma^-)$, where
$U_{\delta_{ij}}=U_{\beta_{-j}}U_{\alpha_i}=U_\beta^{-j}U_\alpha^i$, the proof of
Theorem~\ref{th:main} is immediately reduced to the proof of the following partial case
(treated for $\gamma=0$ in \cite[Theorem~5.2]{KKL14}).
\begin{theorem}\label{th:one-shift}
Suppose $1<p<\infty$ and $\gamma\in\mC$ is such that $0<1/p+\Re\gamma<1$.
If $\alpha\in SOS(\mR_+)$ and
\begin{equation}\label{eq:main-condition}
0<\frac{1}{p}+\Re\gamma+\frac{1}{2\pi}\inf_{t\in\mR_+}(\omega(t)\Im\gamma),
\quad
\frac{1}{p}+\Re\gamma+\frac{1}{2\pi}\sup_{t\in\mR_+}(\omega(t)\Im\gamma)<1,
\end{equation}
where $\omega(t)=\log[\alpha(t)/t]$,
then the operators $U_\alpha P_\gamma^+ +P_\gamma^-$
and $U_\alpha^{-1}P_{\overline{\gamma}}^++P_{\overline{\gamma}}^-$
are Fredholm on the space $L^p(\mR_+)$ and
\[
\Ind(U_\alpha P_\gamma^++P_\gamma^-)
=
\Ind(U_\alpha^{-1}P_{\overline{\gamma}}^++P_{\overline{\gamma}}^-)
=0.
\]
\end{theorem}
The rest of the paper is devoted to the proof of Theorem~\ref{th:one-shift}. In
Section~\ref{sec:reduction} we reduce the proof of Theorem~\ref{th:one-shift} to the study
of the Fredholmness and the calculation of the index of an operator $A_{\alpha,\gamma}$
involving the operators $U_\alpha$, $U_\alpha^{-1}$ and $R_\gamma$, $R_{\overline{\gamma}}$,
where
\[
(R_\gamma f)(t):=\frac{1}{\pi i}\int_{\mR_+}
\left(\frac{t}{\tau}\right)^\gamma\frac{f(\tau)}{\tau+t}d\tau,
\quad
\gamma\in\mC,\quad 0<1/p+\Re\gamma<1.
\]
This reduction is based on the well known fact that $S_\gamma$ and $R_\gamma$ with
$\gamma\in\mC$ satisfying $0<1/p+\Re\gamma<1$ are similar to certain Mellin convolution
operators. To study the operator $A_{\alpha,\gamma}$, we need a more sophisticated machinery
of Mellin pseudodifferential operators. In Section~\ref{sec:Mellin-PDO} we formulate several
results on Mellin pseudodifferential operators concerning their boundedeness, compactness,
and Fredholmness. In Section~\ref{sec:similarity} we prove that the operator
$A_{\alpha,\gamma}$ is similar to a Mellin pseudodifferential operator $\Op(\fg_{\omega,\gamma})$
with certain  explicitly given symbol $\fg_{\omega,\gamma}$. At this step the proof of
Theorem~\ref{th:one-shift} is reduced to the study of the operator $\Op(\fg_{\omega,\gamma})$.
The important point is that the Mellin pseudodifferential operator $\Op(\fg_{\omega,\gamma})$
belongs to the class $\{\Op(\fa):\fa\in\widetilde{\cE}(\mR_+,V(\mR))\}$, for which a Fredholm
criterion and an index formula are available. Here $\widetilde{\cE}(\mR_+,V(\mR))$ is the Banach
algebra of slowly oscillating symbols of limited smoothness introduced by Yuri Karlovich in
\cite{K09}. In Section~\ref{sec:Fredholmness-of-my-PDO} we show that $\Op(\fg_{\omega,\gamma})$
is Fredholm and $\Ind\Op(\fg_{\omega,\gamma})=0$. This completes the proof of
Theorem~\ref{th:one-shift} and, thus, of Theorem~\ref{th:main}.
\section{Reduction}\label{sec:reduction}
In this section we introduce an operator $A_{\alpha,\gamma}$ involving the operators
$U_\alpha$, $U_\alpha^{-1}$, $R_\gamma$, and $R_{\overline{\gamma}}$. We show that if the
operator $A_{\alpha,\gamma}$ is Fredholm and its index is equal to zero, then the operators
$U_\alpha P_\gamma^++P_\gamma^-$ and $U_\alpha^{-1}P_{\overline{\gamma}}^++P_{\overline{\gamma}}^-$
are so.
\subsection{Fourier and Mellin convolution operators}
Let $\cF:L^2(\mR)\to L^2(\mR)$ denote the Fourier transform,
\[
(\cF f)(x):=\int_\mR f(y)e^{-ixy}dy,\quad x\in\mR,
\]
and let $\cF^{-1}:L^2(\mR)\to L^2(\mR)$ be the inverse of $\cF$. A function
$a\in L^\infty(\mR)$ is called a Fourier multiplier on $L^p(\mR)$
if the mapping
$f\mapsto \cF^{-1}a\cF f$ maps $L^2(\mR)\cap L^p(\mR)$ onto itself and extends
to a bounded operator on $L^p(\mR)$. The latter operator is then denoted by
$W^0(a)$. We let $\cM_p(\mR)$ stand for the set of all Fourier multipliers on
$L^p(\mR)$. One can show that $\cM_p(\mR)$ is a Banach algebra under the norm
\[
\|a\|_{\cM_p(\mR)}:=\|W^0(a)\|_{\cB(L^p(\mR))}.
\]

Let $d\mu(t)=dt/t$ be the (normalized) invariant measure on $\mR_+$.
Consider the Fourier transform on $L^2(\mathbb{R}_+,d\mu)$, which is
usually referred to as the Mellin transform and is defined by
\[
\cM:L^2(\mR_+,d\mu)\to L^2(\mR),
\quad
(\cM f)(x):=\int_{\mR_+} f(t) t^{-ix}\,\frac{dt}{t}.
\]
It is an invertible operator, with inverse given by
\[
{\cM^{-1}}:L^2(\mR)\to L^2(\mR_{+},d\mu),
\quad
({\cM^{-1}}g)(t)= \frac{1}{2\pi}\int_{\mR}
g(x)t^{ix}\,dx.
\]
Let $E$ be the isometric isomorphism
\begin{equation}\label{eq:def-E}
E:L^p(\mR_+,d\mu)\to L^p(\mR),
\quad
(Ef)(x):=f(e^x),\quad x\in\mR.
\end{equation}
Then the map
\[
A\mapsto E^{-1}AE
\]
transforms the Fourier convolution
operator $W^0(a)=\cF^{-1}a\cF$ to the Mellin convolution operator
\[
\operatorname{Co}(a):=\cM^{-1}a\cM
\]
with the same symbol $a$. Hence the class of Fourier multipliers on
$L^p(\mR)$ coincides with the class of Mellin multipliers on $L^p(\mR_+,d\mu)$.
\subsection{Some operator identities}
Let $\cA$ be the smallest closed subalgebra of $\cB(L^p(\mR_+))$ that contains the operators
$I$ and $S_0$.
\begin{lemma}[{\cite[Lemma~2.7]{KKL14}}]
\label{le:compactness-commutators}
If $\alpha\in SOS(\mR_+)$, $A\in\cA$, then $U_\alpha^{\pm 1} A\simeq AU_\alpha^{\pm 1}$.
\end{lemma}
Consider the isometric isomorphism
\[
\Phi:L^p(\mR_+)\to L^p(\mR_+,d\mu),
\quad
(\Phi f)(t):=t^{1/p}f(t),\quad t\in\mR_+.
\]
The following two statements are well known and go back to Duduchava \cite{D79,D87}  (see also
\cite[Section~1.10.2]{DS08},
\cite[Section~2.1.2]{HRS94}, and
\cite[Sections~4.2.2--4.2.3]{RSS11}).
\begin{lemma}\label{le:alg-A-commutative}
The algebra $\cA$ is commutative.
\end{lemma}
\begin{lemma}\label{le:alg-A}
Let $1<p<\infty$ and $\gamma\in\mC$ be such that $0<1/p+\Re\gamma<1$.
The functions $s_\gamma$ and $r_\gamma$, given for $x\in\mR$ by
\[
s_\gamma(x):=\coth[\pi(x+i/p+i\gamma)],
\quad
r_\gamma(x):=1/\sinh[\pi(x+i/p+i\gamma)],
\]
belong to $\cM_p(\mR)$; the operators $S_\gamma$ and $R_\gamma$ belong to $\cA$; and
\[
S_\gamma=\Phi^{-1}\Co(s_\gamma)\Phi,
\quad
R_\gamma=\Phi^{-1}\Co(r_\gamma)\Phi.
\]
\end{lemma}
The above lemma is a source for various relations involving operators $S_\gamma$ and $R_\gamma$.
We will need the following two identities, which we were unable to find in the literature
(although similar results and/or techniques were used, for instance, in \cite{D87}, \cite[p.~50]{HRS94}).
\begin{lemma}\label{le:P-pm-relations}
Let $1<p<\infty$ and $\gamma,\delta\in\mC$ be such that $0<1/p+\Re\gamma<1$ and
$0<1/p+\Re\delta<1$. Then
\begin{equation}\label{eq:P-pm-relations-1}
P_\gamma^\pm P_\delta^\pm=
\frac{1}{2}P_\gamma^\pm+\frac{1}{2}P_\delta^\pm+
\frac{\cos[\pi(\gamma-\delta)]}{4}R_\gamma R_\delta,
\quad
P_\gamma^- P_\delta^+ =-\frac{e^{i\pi(\delta-\gamma)}}{4}R_\gamma R_\delta.
\end{equation}
\end{lemma}
\begin{proof}
Let $x\in\mR$. Put
\[
x_\gamma:=\pi(x+i/p+i\gamma),
\quad
x_\delta:=\pi(x+i/p+i\delta).
\]
It is easy to see that
\begin{align}
s_\gamma(x)s_\delta(x)-1
&=
\frac{\cosh(x_\gamma)\cosh(x_\delta)-\sinh(x_\gamma)\sinh(x_\delta)}
{\sinh(x_\gamma)\sinh(x_\delta)}
=
\frac{\cosh(x_\gamma-x_\delta)}{\sinh(x_\gamma)\sinh(x_\delta)}
\nonumber\\
&=
\cosh[i\pi(\gamma-\delta)]r_\gamma(x)r_\delta(x)
=
\cos[\pi(\gamma-\delta)]r_\gamma(x)r_\delta(x)
\label{eq:P-pm-relations-2}
\end{align}
and
\begin{align}
s_\gamma(x)-s_\delta(x)
&=
\frac{\cosh(x_\gamma)\sinh(x_\delta)-\cosh(x_\delta)\sinh(x_\gamma)}{\sinh(x_\gamma)\sinh(x_\delta)}
=
\frac{\sinh(x_\delta-x_\gamma)}{\sinh(x_\gamma)\sinh(x_\delta)}
\nonumber\\
&=
\sinh[i\pi(\delta-\gamma)]r_\gamma(x)r_\delta(x)
=
-i\sin[\pi(\gamma-\delta)]r_\gamma(x)r_\delta(x).
\label{eq:P-pm-relations-3}
\end{align}
Put
\begin{equation}\label{eq:P-pm-relations-4}
p_\gamma^\pm (x):=(1\pm s_\gamma(x))/2,
\quad
p_\delta^\pm(x):=(1\pm s_\delta(x))/2.
\end{equation}
Taking into account \eqref{eq:P-pm-relations-2}, we obtain
\begin{align}
p_\gamma^\pm(x)p_\delta^\pm(x)
&=
\frac{1}{4}(1\pm s_\gamma(x)\pm s_\delta(x)+s_\gamma(x)s_\delta(x))
\nonumber\\
&=
\frac{1}{2}p_\gamma^\pm(x)+\frac{1}{2}p_\delta^\pm(x)+\frac{1}{4}(s_\gamma(x)s_\delta(x)-1)
\nonumber\\
&=
\frac{1}{2}p_\gamma^\pm(x)+\frac{1}{2}p_\delta^\pm(x)+\frac{\cos[\pi(\gamma-\delta)]}{4}r_\gamma(x)r_\delta(x).
\label{eq:P-pm-relations-5}
\end{align}
From \eqref{eq:P-pm-relations-2}--\eqref{eq:P-pm-relations-3} we get
\begin{align}
p_\gamma^-(x)p_\delta^+(x)
&=
\frac{1}{4}(1+s_\delta(x)-s_\gamma(x)-s_\gamma(x)s_\delta(x))
\nonumber\\
&=
-\frac{1}{4}(s_\gamma(x)-s_\delta(x))-\frac{1}{4}(s_\gamma(x)s_\delta(x)-1)
\nonumber\\
&=
-\frac{1}{4}(-i\sin[\pi(\gamma-\delta)]r_\gamma(x)r_\delta(x))-\frac{\cos[\pi(\gamma-\delta)]}{4}r_\gamma(x)r_\delta(x)
\nonumber\\
&=
-\frac{1}{4}(\cos[\pi(\delta-\gamma)]+i\sin[\pi(\delta-\gamma)])r_\gamma(x)r_\delta(x)
\nonumber\\
&=
-\frac{e^{i\pi(\delta-\gamma)}}{4}r_\gamma(x)r_\delta(x).
\label{eq:P-pm-relations-6}
\end{align}
Combining identities \eqref{eq:P-pm-relations-4}--\eqref{eq:P-pm-relations-6} with
Lemma~\ref{le:alg-A}, we arrive at \eqref{eq:P-pm-relations-1}.
\end{proof}
As it is pointed out by the anonymous referee, the results of the above lemma can also be derived
by the completely different techniques of Duduchava, Kvergeli\-dze, Tsaava \cite[Theorem~2.2]{DKT13}
based on the Poincare-Beltrami and Muskhelishvili formulas.
\begin{lemma}\label{le:complex-conjugation}
Let $1<p<\infty$ and $\cC$ be the operator of complex conjugation given on $L^p(\mR_+)$ by
\[
(\cC f)(t):=\overline{f(t)}, \quad t\in\mR_+.
\]
If $\alpha\in SOS(\mR_+)$ and $\gamma\in\mC$ is such that $0<1/p+\Re\gamma<1$, then
\[
\cC U_\alpha\cC=U_\alpha,
\quad
\cC S_\gamma \cC=-S_{\overline{\gamma}},
\quad
\cC P_\gamma^+\cC = P_{\overline{\gamma}}^-,
\quad
\cC P_\gamma^-\cC = P_{\overline{\gamma}}^+.
\]
\end{lemma}
\begin{proof}
The first identity is trivial. The second equality follows from $\overline{1/(\pi i)}=-1/(\pi i)$
and $\overline{(t/\tau)^\gamma}=(t/\tau)^{\overline{\gamma}}$, where $t,\tau\in\mR_+$. The
remaining identities are immediate corollaries of the second equality.
\end{proof}
\subsection{Starting the proof of Theorem~\ref{th:one-shift}}
Now we are in a position to prove the main result of this section.
\begin{theorem}\label{th:reduction}
Suppose $1<p<\infty$ and $\gamma\in\mC$ is such that $0<1/p+\Re\gamma<1$ and $\alpha\in SOS(\mR_+)$.
If the operator
\begin{equation}\label{eq:operator-A}
A_{\alpha,\gamma}:=I+
\frac{1}{4}\big[e^{2\pi\Im\gamma}(I-U_\alpha^{-1})+
e^{-2\pi\Im\gamma}(I-U_\alpha)\big]
R_\gamma R_{\overline{\gamma}}.
\end{equation}
is Fredholm on the space $L^p(\mR_+)$ and $\Ind A_{\alpha,\gamma}=0$,
then the operators $U_\alpha P_\gamma^++P_\gamma^-$ and
$U_\alpha^{-1}P_{\overline{\gamma}}^++P_{\overline{\gamma}}^-$
are Fredholm on the space $L^p(\mR_+)$ and
\begin{equation}\label{eq:reduction-0}
\Ind(U_\alpha P_\gamma^++P_\gamma^-)
=
\Ind(U_\alpha^{-1}P_{\overline{\gamma}}^++P_{\overline{\gamma}}^-)
=0.
\end{equation}
\end{theorem}
\begin{proof}
The idea of the proof is borrowed from \cite[Theorem~9.4]{K07}
(see also \cite[Theorem~6.1]{K08}).
From Lemmas~\ref{le:alg-A-commutative} and \ref{le:alg-A} it follows that the operators
$P_\gamma^+$, $P_\gamma^-$, $P_{\overline{\gamma}}^+$, $P_{\overline{\gamma}}^-$,
$R_\gamma,R_{\overline{\gamma}}$ pairwise commute. From this observation and
Lemma~\ref{le:compactness-commutators}  we obtain
\begin{align}
(U_\alpha P_\gamma^+ +P_\gamma^-)(U_\alpha^{-1}P_{\overline{\gamma}}^++P_{\overline{\gamma}}^-)
& \simeq
\widetilde{A}_{\alpha,\gamma},
\label{eq:reduction-1}
\\
(U_\alpha^{-1}P_{\overline{\gamma}}^++P_{\overline{\gamma}}^-)(U_\alpha P_\gamma^+ +P_\gamma^-)
& \simeq
\widetilde{A}_{\alpha,\gamma}.
\label{eq:reduction-2}
\end{align}
where
\[
\widetilde{A}_{\alpha,\gamma}:=
P_\gamma^+ P_{\overline{\gamma}}^+
+
U_\alpha^{-1}P_\gamma^- P_{\overline{\gamma}}^+
+
U_\alpha P_{\overline{\gamma}}^- P_\gamma^+
+
P_\gamma^- P_{\overline{\gamma}}^-.
\]
Since $\gamma-\overline{\gamma}=2i\Im\gamma$, we have
\[
\cos[\pi(\gamma-\overline{\gamma})]=\frac{e^{2\pi\Im\gamma}+e^{-2\pi\Im\gamma}}{2},
\quad
e^{i\pi(\gamma-\overline{\gamma})}=e^{-2\pi\Im\gamma},
\quad
e^{i\pi(\overline{\gamma}-\gamma)}=e^{2\pi\Im\gamma}.
\]
From Lemma~\ref{le:P-pm-relations} and the above identities we get
\begin{align}
\widetilde{A}_{\alpha,\gamma}
= &
\left(\frac{1}{2}P_\gamma^++\frac{1}{2}P_{\overline{\gamma}}^+
+\frac{\cos[\pi(\gamma-\overline{\gamma})]}{4}R_\gamma R_{\overline{\gamma}}\right)
-\frac{e^{i\pi(\overline{\gamma}-\gamma)}}{4}U_\alpha^{-1}R_\gamma R_{\overline{\gamma}}
\nonumber\\
&+
\left(\frac{1}{2}P_\gamma^-+\frac{1}{2}P_{\overline{\gamma}}^-
+\frac{\cos[\pi(\gamma-\overline{\gamma})]}{4}R_\gamma R_{\overline{\gamma}}\right)
-\frac{e^{i\pi(\gamma-\overline{\gamma})}}{4}U_\alpha R_\gamma R_{\overline{\gamma}}
\nonumber\\
=&
I+\frac{e^{2\pi\Im\gamma}+e^{-2\pi\Im\gamma}}{4}R_\gamma R_{\overline{\gamma}}
-\frac{e^{2\pi\Im\gamma}}{4}U_\alpha^{-1}R_\gamma R_{\overline{\gamma}}
-\frac{e^{-2\pi\Im\gamma}}{4}U_\alpha R_\gamma R_{\overline{\gamma}}
\nonumber\\
=& A_{\alpha,\gamma}.
\label{eq:reduction-3}
\end{align}
By the hypothesis, $A_{\alpha,\gamma}$ is Fredholm. Therefore there exist $B_r,B_l\in\cB(L^p(\mR_+))$ such
that
\begin{equation}\label{eq:reduction-4}
A_{\alpha,\gamma}B_r\simeq B_lA_{\alpha,\gamma}\simeq I.
\end{equation}
Multiplying \eqref{eq:reduction-1} from the right by $B_r$ and from the left by $B_l$
and taking into account \eqref{eq:reduction-3}--\eqref{eq:reduction-4}, we see that
$(U_\alpha^{-1}P_{\overline{\gamma}}^++P_{\overline{\gamma}}^-)B_r$ is a right regularizer
for the operator $U_\alpha P_\gamma^++P_\gamma^-$ and $B_l(U_\alpha P_\gamma^++P_\gamma^-)$
is a left regularizer for the operator $U_\alpha^{-1}P_{\overline{\gamma}}^++P_{\overline{\gamma}}^-$.
Analogously, from \eqref{eq:reduction-2}--\eqref{eq:reduction-4} we deduce that
$B_l(U_\alpha^{-1}P_{\overline{\gamma}}^++P_{\overline{\gamma}}^-)$ is a left regularizer
for $U_\alpha P_\gamma^++P_\gamma^-$ and $(U_\alpha P_\gamma^++P_\gamma^-)B_r$ is a right
regularizer for $U_\alpha^{-1}P_{\overline{\gamma}}^++P_{\overline{\gamma}}^-$.
Thus, both $U_\alpha P_\gamma^++P_\gamma^-$ and
$U_\alpha^{-1}P_{\overline{\gamma}}^++P_{\overline{\gamma}}^-$ are Fredholm.

From this observation, relations \eqref{eq:reduction-1}, \eqref{eq:reduction-3},
and well known properties of indices of Fredholm operators we derive that
\begin{equation}\label{eq:reduction-5}
\Ind(U_\alpha P_\gamma^++P_\gamma^-)+
\Ind(U_\alpha^{-1}P_{\overline{\gamma}}^++P_{\overline{\gamma}}^-)=
\Ind A_{\alpha,\gamma}=0.
\end{equation}
Since the operator of complex conjugation $\cC$ is isometric and anti-linear on $L^p(\mR_+)$,
we have
\begin{equation}\label{eq:reductiojn-6}
\Ind(U_\alpha P_\gamma^++P_\gamma^-)
=
\Ind[\cC(U_\alpha P_\gamma^++P_\gamma^-)\cC].
\end{equation}
From Lemma~\ref{le:complex-conjugation} we obtain
\[
\cC(U_\alpha P_\gamma^++P_\gamma^-)\cC
=
U_\alpha P_{\overline{\gamma}}^-+P_{\overline{\gamma}}^+
=
U_\alpha (U_\alpha^{-1}P_{\overline{\gamma}}^++P_{\overline{\gamma}}^-),
\]
whence
\begin{equation}\label{eq:reduction-7}
\Ind[\cC(U_\alpha P_\gamma^++P_\gamma^-)\cC]
=
\Ind U_\alpha+\Ind (U_\alpha^{-1}P_{\overline{\gamma}}^++P_{\overline{\gamma}}^-)
=
\Ind (U_\alpha^{-1}P_{\overline{\gamma}}^++P_{\overline{\gamma}}^-).
\end{equation}
Combining \eqref{eq:reduction-5}--\eqref{eq:reduction-7}, we arrive at \eqref{eq:reduction-0}.
\end{proof}
Due to Theorem~\ref{th:reduction}, in order to complete the proof of
Theorem~\ref{th:one-shift}, it remains to show that the operator $A_{\alpha,\gamma}$ given by
\eqref{eq:operator-A} is Fredholm and $\Ind A_{\alpha,\gamma}=0$.
\section{Mellin pseudodifferential operators and their symbols}
\label{sec:Mellin-PDO}
\subsection{Melin PDO's: overview}
Mellin pseudodifferential operators are generalizations of Mellin convolution
operators. Let $\fa$ be a sufficiently smooth function defined on $\mR_+\times\mR$.
The Mellin pseudodifferential operator (shortly, Mellin PDO) with symbol $\fa$
is initially defined for smooth functions $f$ of compact support by the iterated
integral
\begin{align}
[\Op(\fa) f](t)
&=
[\cM^{-1}\fa(t,\cdot)\cM f](t)
\nonumber\\
&=
\frac{1}{2\pi}\int_\mR dx \int_{\mR_+}
\fa(t,x)\left(\frac{t}{\tau}\right)^{ix}f(\tau) \frac{d\tau}{\tau}
\quad\mbox{for}\quad t\in\mR_+.
\label{eq:Mellin-PDO}
\end{align}
Obviously, if $\fa(t,x)=a(x)$ for all $(t,x)\in\mR_+\times\mR$, then the Mellin
pseudodifferential operator $\Op(\fa)$ becomes the Mellin convolution operator
$\Op(\fa)=\Co(a)$.

In 1991 Vladi\-mir Rabinovich \cite{R92} proposed to use Mellin pseudodifferential
operators (shortly, Mellin PDO's) with $C^\infty$ slowly oscillating
symbols to study singular integral operators with slowly oscillating coefficients
on $L^p$ spaces. This idea was exploited in a series of papers by Rabinovich and
coauthors (see \cite[Sections~4.6--4.7]{RRS04} for the complete history up to 2004).
Rabinovich stated in \cite[Theorem~2.6]{R98} a Fredholm criterion for Mellin PDO's
with $C^\infty$ slowly oscillating (or slowly varying) symbols
on the spaces $L^p(\mR_+,d\mu)$ for $1<p<\infty$. Namely,
he considered symbols $\fa\in C^\infty(\mR_+\times\mR)$ such that
\begin{equation}\label{eq:Hoermander}
\sup_{(t,x)\in\mR_+\times\mR}
\big|(t\partial_t)^j\partial_x^k\fa(t,x)\big|(1+x^2)^{k/2}<\infty
\quad\mbox{for all}\quad j,k\in\mZ_+
\end{equation}
and
\begin{equation}\label{eq:Grushin}
\lim_{t\to s}\sup_{x\in\mR}
\big|(t\partial_t)^j\partial_x^k\fa(t,x)\big|(1+x^2)^{k/2}=0
\mbox{ for all }\ j\in\mN,\ k\in\mZ_+, \ s\in\{0,\infty\}.
\end{equation}
Here and in what follows $\partial_t$ and $\partial_x$ denote the operators of
partial  differentiation with respect to $t$ and to $x$. Notice that
\eqref{eq:Hoermander}  defines nothing but the Mellin version of the H\"ormander
class $S_{1,0}^0(\mR)$ (see, e.g., \cite{H67}, \cite[Chap.~2, Section~1]{K82} for
the definition of the  H\"ormander classes $S_{\varrho,\delta}^m(\mR^n)$). If
$\fa$ satisfies  \eqref{eq:Hoermander}, then the Mellin PDO $\Op(\fa)$ is bounded
on the spaces $L^p(\mR_+,d\mu)$ for $1<p<\infty$ (see, e.g.,
\cite[Chap.~VI, Proposition~4]{St93} for the corresponding Fourier PDO's).
Condition \eqref{eq:Grushin}
is the Mellin version of Grushin's definition of slowly varying symbols in the
first variable (see, e.g., \cite{G70}, \cite[Chap.~3, Defintion~5.11]{K82}).

The above mentioned results have a disadvantage that the smoothness conditions
imposed on slowly oscillating symbols are very strong. In particular, they are not
applicable directly to the problem we are considering in the paper.

In 2005 Yuri Karlovich \cite{K06} (see also \cite{K06-IWOTA,K09}) developed a
Fredholm theory of Fourier pseudodifferential with slowly oscillating symbols
of limited smoothness in the spirit of Sarason's definition \cite[p.~820]{S77}
of slow oscillation adopted in the present paper (much less restrictive than in
\cite{R98} and in the works mentioned in \cite{RRS04}).
Necessary for our purposes results from \cite{K06,K06-IWOTA,K09} were translated to
the Mellin setting, for instance, in \cite{KKL14} with the aid of the transformation
\[
A\mapsto E^{-1}AE,
\]
where $A\in\cB(L^p(\mR))$ and the isometric isomorphism $E:L^p(\mR_+,d\mu)\to L^p(\mR)$
is defined by \eqref{eq:def-E}. For the convenience of the reader, we reproduce these
results here exactly in the same form as they were stated in \cite{KKL14}, where more
details on their proofs can be found.
\subsection{Boundedness of Mellin PDO's}
Let $a$ be an absolutely continuous function of finite total variation
\[
V(a):=\int_\mR|a'(x)|dx
\]
on $\mR$. The set $V(\mR)$ of all absolutely continuous functions of finite
total variation on $\mR$ becomes a Banach algebra equipped with the norm
\[
\|a\|_V:=\|a\|_{L^\infty(\mR)}+V(a).
\]

Let $C_b(\mR_+,V(\mR))$ denote the Banach
algebra of all bounded continuous $V(\mR)$-valued functions on $\mR_+$ with
the norm
\[
\|\fa(\cdot,\cdot)\|_{C_b(\mR_+,V(\mR))}
=
\sup_{t\in\mR_+}\|\fa(t,\cdot)\|_V.
\]
As usual, let $C_0^\infty(\mR_+)$ be the set of all infinitely differentiable
functions of compact support on $\mR_+$.
\begin{theorem}[{\cite[Theorem~3.1]{KKL14}}]
\label{th:boundedness-PDO}
If $\fa\in C_b(\mR_+,V(\mR))$, then the Mellin pseudodifferential operator
$\operatorname{Op}(\fa)$, defined for functions $f\in C_0^\infty(\mR_+)$ by
the iterated integral \eqref{eq:Mellin-PDO},
extends to a bounded linear operator on the space $L^p(\mR_+,d\mu)$
and there is a number $C_p\in(0,\infty)$ depending only on $p$ such
that
\[
\|\operatorname{Op}(\fa)\|_{\cB(L^p(\mR_+,d\mu))}
\le C_p\|\fa\|_{C_b(\mR_+,V(\mR))}.
\]
\end{theorem}
\subsection{Compactness of Mellin PDO's and their semi-commutators}
Consider the Banach subalgebra $SO(\mR_+,V(\mR))$ of the algebra $C_b(\mR_+,V(\mR))$
consisting of all $V(\mR)$-valued functions $\fa$ on $\mR_+$ that slowly
oscillate at $0$ and $\infty$, that is,
\[
\lim_{r\to s} \max_{t,\tau\in[r,2r]}\|\fa(t,\cdot)-\fa(\tau,\cdot)\|_{L^\infty(\mR)}=0,
\quad s\in\{0,\infty\}.
\]

Let $\cE(\mR_+,V(\mR))$ be the Banach algebra of all $V(\mR)$-valued functions
$\fa$ in the algebra $SO(\mR_+,V(\mR))$ such that
\[
\lim_{|h|\to 0}\sup_{t\in\mR_+}\big\|\fa(t,\cdot)-\fa^h(t,\cdot)\big\|_V=0
\]
where $\fa^h(t,x):=\fa(t,x+h)$ for all $(t,x)\in\mR_+\times \mR$.
\begin{theorem}[{\cite[Theorem~3.2]{KKL14}}]
\label{th:compactness-PDO}
If $\fa\in\cE(\mR_+,V(\mR))$ and
\[
\lim_{\ln^2 t+x^2\to\infty}\fa(t,x)=0,
\]
then $\Op(\fa)\in\cK(L^p(\mR_+,d\mu))$.
\end{theorem}
\begin{theorem}[{\cite[Theorem~3.3]{KKL14}}]
\label{th:comp-semi-commutators-PDO}
If $\fa,\fb\in\cE(\mR_+,V(\mR))$, then
\[
\operatorname{Op}(\fa)\operatorname{Op}(\fb)\simeq \operatorname{Op}(\fa\fb).
\]
\end{theorem}
\subsection{Fredholmness of Mellin PDO's}
For a unital commutative Banach algebra $\fA$, let $M(\mathfrak{A})$
denote its maximal ideal space. Identifying the points
$t\in\overline{\mR}_+$ with the evaluation functionals $t(f)=f(t)$
for $f\in C(\overline{\mR}_+)$, we get
\[
M(C(\overline{\mR}_+))=\overline{\mR}_+.
\]
Consider the fibers
\[
M_s(SO(\mR_+))
:=
\big\{\xi\in M(SO(\mR_+)):\xi|_{C(\overline{\mR}_+)}=s\big\}
\]
of the maximal ideal space $M(SO(\mR_+))$ over the points
$s\in\{0,\infty\}$. By \cite[Proposition~2.1]{K09}, the set
\[
\Delta:=M_0(SO(\mR_+))\cup M_\infty(SO(\mR_+))
\]
coincides with $(\operatorname{clos}_{SO^*}\mR_+)\setminus\mR_+$
where $\operatorname{clos}_{SO^*}\mR_+$ is the weak-star closure
of $\mR_+$ in the dual space of $SO(\mR_+)$. Then $M(SO(\mR_+))
=\Delta\cup\mR_+$.

Let $\fa\in\cE(\mR_+,V(\mR))$. For every $t\in\mR_+$, the function $\fa(t,\cdot)$
belongs to $V(\mR)$ and, therefore, has finite limits at $\pm\infty$, which will
be denoted by $\fa(t,\pm\infty)$. Now we explain how to extend the function
$\fa$ to $\Delta\times\overline{\mR}$.
\begin{lemma}[{\cite[Lemma~3.5]{KKL14}}]
\label{le:values}
Let $s\in\{0,\infty\}$ and $\fa\in\cE(\mR_+,V(\mR))$. For each $\xi\in M_s(SO(\mR_+))$
there is a sequence $\{t_n\}\subset\mR_+$ such that $t_n\to s$ as $n\to\infty$
and a function $\fa(\xi,\cdot)\in V(\mR)$ such that
\[
\fa(\xi,x)=\lim_{n\to\infty}\fa(t_n,x)\quad\mbox{for every}\quad x\in\overline{\mR}.
\]
\end{lemma}
To study the Fredholmness of Mellin pseudodifferential operators, we need the
Banach algebra $\widetilde{\cE}(\mR_+,V(\mR))$ consisting of all functions
$\fa$ belonging to $\cE(\mR_+,V(\mR))$ and such that
\[
\lim_{m\to\infty}\sup_{t\in\mR_+}\int_{\mR\setminus[-m,m]}
\left|\partial_x \fa(t,x)\right|\,dx=0.
\]
\begin{theorem}[{\cite[Theorem~3.6]{KKL14}}]
\label{th:Fredholmness-PDO}
If $\fa\in\widetilde{\cE}(\mR_+,V(\mR))$, then the Mellin pseudodifferential
operator $\operatorname{Op}(\fa)$ is Fredholm on the space $L^p(\mR_+,d\mu)$
if and only if
\begin{equation}\label{eq:Fredholmness-PDO}
\fa(t,\pm\infty)\ne 0
\ \text{ for all }\
t\in\mR_+,
\quad
\fa(\xi,x)\ne 0
\ \text{ for all }\
(\xi,x)\in\Delta\times\overline{\mR}.
\end{equation}
In the case of Fredholmness
\[
\Ind \operatorname{Op}(\fa)
=
\lim_{\tau\to+\infty}\frac{1}{2\pi}
\big\{\arg\fa(t,x)\big\}_{(t,x)\in\partial\Pi_\tau},
\]
where $\Pi_\tau=[\tau^{-1},\tau]\times\overline{\mR}$ and
$\{\arg\fa(t,x)\}_{(t,x)\in\partial\Pi_\tau}$ denotes the increment of the function
$\arg\fa(t,x)$ when the point $(t,x)$ traces the boundary $\partial\Pi_\tau$ of
$\Pi_\tau$ counter-clockwise.
\end{theorem}
This result follows from \cite[Theorem~4.3]{K09}. Note that for infinitely differentiable
slowly oscillating symbols such result was obtained earlier in \cite[Theorem~2.6]{R98}.
\section{The operator $A_{\alpha,\gamma}$ is similar to a Mellin PDO $\Op(\fg_{\omega,\gamma})$ \\
with symbol in the algebra $\widetilde{\cE}(\mR_+,V(\mR))$}
\label{sec:similarity}
In this section we show that the operator $A_{\alpha,\gamma}$ given by \eqref{eq:operator-A}
satisfies the relation $A_{\alpha,\gamma}\simeq\Phi^{-1}\Op(\fg_{\omega,\gamma})\Phi$, where
$\fg_{\omega,\gamma}$ is an explicitly given function in the algebra $\widetilde{\cE}(\mR_+,V(\mR))$.
Hence the operator $A_{\alpha,\gamma}$ is Fredholm if and only if $\Op(\fg_{\omega,\gamma})$
is Fredholm and $\Ind A_{\alpha,\gamma}=\Ind\Op(\fg_{\omega,\gamma})$.
\subsection{Some important functions in the algebra $\widetilde{\cE}(\mR_+,V(\mR))$}
The following statement for $\gamma=0$ was proved in \cite[Lemma~7.1]{KKL11a} and
\cite[Lemma~4.2]{KKL14}. For $\gamma\in\mC$ satisfying $0<1/p+\Re\gamma<1$ the proof is the same.
\begin{lemma}\label{le:ff-fs-fr}
Suppose $f\in SO(\mR_+)$, $1<p<\infty$, and $\gamma\in\mC$ is such that $0<1/p+\Re\gamma<1$.
Then the functions
\[
\ff(t,x):=f(t),
\quad
\fs_\gamma(t,x):=s_\gamma(x),
\quad
\fr_\gamma(t,x):=r_\gamma(x),
\quad
(t,x)\in\mR_+\times\mR,
\]
belong to the Banach algebra $\widetilde{\cE}(\mR_+,V(\mR))$.
\end{lemma}
The next lemma for $\gamma=0$ was proved in \cite[Lemma~4.3]{KKL14} (see also \cite[Lemmas~7.3--7.4]{KKL11a}).
For arbitrary $\gamma\in\mC$ satisfying $0<1/p+\Re\gamma<1$ the arguments remain unalterated.
\begin{lemma}\label{le:fb}
Suppose $\omega\in SO(\mR_+)$ is a real-valued function, $1<p<\infty$, and $\gamma\in\mC$ is
such that $0<1/p+\Re\gamma<1$. Then the function
\[
\fb_{\omega,\gamma}(t,x):=e^{i\omega(t)x}r_\gamma(x),
\quad (t,x)\in\mR_+\times\mR,
\]
belongs to the Banach algebra $\widetilde{\cE}(\mR_+,V(\mR))$ and there
is a positive constant $C(p,\gamma)$ depending only on $p$ and $\gamma$ such that
\[
\|\fb_{\omega,\gamma}\|_{C_b(\mR_+,V(\mR))}
\le
C(p,\gamma)\left(1+\sup_{t\in\mR_+}|\omega(t)|\right).
\]
\end{lemma}
\subsection{Operators $U_\alpha R_\gamma$ and $U_\alpha^{-1}R_\gamma$}
Let $C^1(\mR_+)$ denote the set of all continuously differentiable functions on $\mR_+$.
We start with the important characterization of orientation-preserving
slowly oscillating shifts.
\begin{lemma}[{\cite[Lemma~2.2]{KKL11a}}]
\label{le:exponent-shift}
An orientation-preserving shift $\alpha:\mR_+\to\mR_+$ belongs to $SOS(\mR_+)$
if and only if
\[
\alpha(t)=te^{\omega (t)},\quad t\in \mR_+,
\]
for some real-valued function $\omega\in SO(\mR_+)\cap C^1(\mR_+)$ such that
$\psi(t):= t\omega^\prime(t)$ also belongs to $SO(\mR_+)$ and
\[
\inf_{t\in\mR_+}\big(1+t\omega'(t)\big)>0.
\]
\end{lemma}
The following statement is crucial for our analysis.
\begin{lemma}\label{le:shift-R-gamma-exact}
Suppose $1<p<\infty$ and $\gamma\in\mC$ is such that $0<1/p+\Re\gamma<1$.
Let $\alpha\in SOS(\mR_+)$ and $U_\alpha$ be the associated isometric shift
operator on $L^p(\mR_+)$. Then the operator $U_\alpha R_\gamma$
can be realized as the Mellin pseudodifferential operator:
\[
U_\alpha R_\gamma = \Phi^{-1}\Op (\fc_{\omega,\gamma}) \Phi,
\]
where the function $\fc_{\omega,\gamma}$, given for $(t,x)\in\mR_+\times\mR$ by
\begin{align}\label{eq:shift-R-gamma-exact-1}
\fc_{\omega,\gamma}(t,x):=(1+t\omega'(t))^{1/p} e^{i\omega(t)x}r_\gamma(x)
\quad\mbox{with}\quad
\omega(t):=\log[\alpha(t)/t],
\end{align}
belongs to the algebra $\widetilde{\cE}(\mR_+,V(\mR))$.
\end{lemma}
\begin{proof}
We follow the proof of \cite[Lemma~4.4]{KKL14}, where this statement was proved
for $\gamma=1/y-1/p$ with $y\in(1,\infty)$ (see also \cite[Lemma~8.3]{KKL11a}).
By Lemma~\ref{le:exponent-shift}, $\alpha(t)=t e^{\omega(t)}$,
where $\omega\in SO(\mR_+)\cap C^1(\mR_+)$ is a real-valued function. Hence
\begin{equation}\label{eq:shift-R-exact-2}
\alpha'(t)=\Omega(t)e^{\omega(t)},
\quad\mbox{where}\quad
\Omega(t):=1+t\omega'(t),
\quad t\in\mR_+.
\end{equation}
Assume that $f\in C_0^\infty(\mR_+)$. Taking into account \eqref{eq:shift-R-exact-2},
we have for $t\in\mR_+$,
\begin{align}
(\Phi U_\alpha R_\gamma\Phi^{-1}f)(t)
&=
\frac{(\alpha'(t))^{1/p}}{\pi i}
\int_{\mR_+}\left(\frac{\alpha(t)}{\tau}\right)^\gamma
\frac{f(\tau)(t/\tau)^{1/p}}{\tau+\alpha(t)}\,d\tau
\nonumber\\
&=
\frac{(\Omega(t))^{1/p}e^{\omega(t)/p}}{\pi i}
\int_{\mR_+}\left(\frac{te^{\omega(t)}}{\tau}\right)^\gamma
\frac{f(\tau)(t/\tau)^{1/p}}{1+e^{\omega(t)}(t/\tau)}\,\frac{d\tau}{\tau}
\nonumber\\
&=
\frac{(\Omega(t))^{1/p}e^{\omega(t)(1/p+\gamma)}}{\pi i}
\int_{\mR_+}
\frac{f(\tau)(t/\tau)^{1/p+\gamma}}{1+e^{\omega(t)}(t/\tau)}\,\frac{d\tau}{\tau}
\nonumber\\
&=(\Omega(t))^{1/p}(I_\gamma f)(t),
\label{eq:shift-R-exact-3}
\end{align}
where
\begin{equation}\label{eq:shift-R-exact-4}
(I_\gamma f)(t):=
\frac{e^{\omega(t)(1/p+\gamma)}}{\pi i}
\int_{\mR_+}
\frac{f(\tau)(t/\tau)^{1/p+\gamma}}{1+e^{\omega(t)}(t/\tau)}\,\frac{d\tau}{\tau}.
\end{equation}
From \cite[formula 6.2.6]{EMOT54} or \cite[formula 3.194.4]{GR07} it follows that for $k>0$,
$\gamma\in\mC$ such that $0<1/p+\Re\gamma<1$, and $x\in\mR$,
\begin{align*}
\frac{1}{\pi i}\int_{\mR_+}\frac{t^{1/p+\gamma}}{1+kt}t^{-ix}\frac{dt}{t}
&=
\frac{1}{\pi i}\cdot\frac{\pi}{k^{1/p+\gamma-ix}}\cdot\frac{1}{\sin[\pi(1/p+\gamma-ix)]}
\\
&=
\frac{1}{k^{1/p+\gamma-ix}}\cdot\frac{1}{\sinh[\pi(x+i/p+i\gamma)]}
\\
&=
e^{i(x+i/p+i\gamma)\log k}r_\gamma(x).
\end{align*}
Taking the inverse Mellin transform, we get
\begin{equation}\label{eq:shift-R-exact-5}
\frac{1}{\pi i}\cdot\frac{t^{1/p+\gamma}}{1+kt}
=
\frac{1}{2\pi}\int_\mR e^{i(x+i/p+i\gamma)\log k}r_\gamma(x)t^{ix}\,dx.
\end{equation}
From \eqref{eq:shift-R-exact-4}--\eqref{eq:shift-R-exact-5} with $k=e^{\omega(t)}$ we obtain
\begin{align}
(I_\gamma f)(t)
&=
\frac{e^{\omega(t)(1/p+\gamma)}}{2\pi}
\int_{\mR_+}
\left(\int_{\mR}e^{i\omega(t)(x+i/p+i\gamma)}r_\gamma(x)\left(\frac{t}{\tau}\right)^{ix}\,dx\right)
f(\tau)\,\frac{d\tau}{\tau}
\nonumber\\
&=
\frac{1}{2\pi}
\int_{\mR_+}
\left(\int_{\mR}e^{i\omega(t)x}r_\gamma(x)\left(\frac{t}{\tau}\right)^{ix}\,dx\right)
f(\tau)\,\frac{d\tau}{\tau}
\nonumber\\
&=
\frac{1}{2\pi}\int_\mR dx
\int_{\mR_+}e^{i\omega(t)x}r_\gamma(x)\left(\frac{t}{\tau}\right)^{ix}f(\tau)\frac{d\tau}{\tau}.
\label{eq:shift-R-exact-6}
\end{align}
From \eqref{eq:shift-R-exact-3} and \eqref{eq:shift-R-exact-6} we obtain for $f\in C_0^\infty(\mR_+)$,
\begin{equation}\label{eq:shift-R-exact-7}
\Phi U_\alpha R_y\Phi^{-1}f=\Op(\fc_{\omega,\gamma})f.
\end{equation}
By Lemma~\ref{le:exponent-shift}, the function $\Omega$ belongs to $SO(\mR_+)$.
Then $\Omega^{1/p}$ is also in $SO(\mR_+)$. Therefore, from Lemmas~\ref{le:ff-fs-fr}
and \ref{le:fb} it follows that the function $\fc_{\omega,\gamma}$ belongs to the
algebra $\widetilde{\cE}(\mR_+,V(\mR))\subset C_b(\mR_+,V(\mR))$. Then
Theorem~\ref{th:boundedness-PDO} implies that $\Op(\fc_{\omega,\gamma})$
extends to a bounded operator on $L^p(\mR_+,d\mu)$. Therefore, from
\eqref{eq:shift-R-exact-7} we obtain
$\Phi U_\alpha R_\gamma\Phi^{-1}=\Op(\fc_{\omega,\gamma})$,
which completes the proof.
\end{proof}
From the above lemma and Theorem~\ref{th:compactness-PDO},
making minor modifications in the proof of \cite[Lemma~4.5]{KKL14},
we can get the following.
\begin{lemma}\label{le:shift-R-gamma}
Suppose $1<p<\infty$ and $\gamma\in\mC$ is such that $0<1/p+\Re\gamma<1$.
Let $\alpha\in SOS(\mR_+)$ and $U_\alpha$ be the associated isometric shift
operator on $L^p(\mR_+)$. Then the operators $U_\alpha R_\gamma$ and
$U_\alpha^{-1}R_\gamma$ can be realized as the Mellin pseudodifferential
operators up to compact operators:
\[
U_\alpha^{\pm 1} R_\gamma \simeq \Phi^{-1}\Op (\fc_{\omega,\gamma}^\pm) \Phi,
\]
where the functions $\fc_{\omega,\gamma}^\pm$, given for $(t,x)\in\mR_+\times\mR$ by
\begin{equation}\label{eq:shift-R-gamma}
\fc_{\omega,\gamma}^\pm(t,x):=e^{\pm i\omega(t)x}r_\gamma(x)
\quad\mbox{with}\quad
\omega(t):=\log[\alpha(t)/t],
\end{equation}
belong to the algebra $\widetilde{\cE}(\mR_+,V(\mR))$.
\end{lemma}
\subsection{Similarity relation}
We are ready to prove the main result of this section.
\begin{theorem}\label{th:similarity}
Suppose $1<p<\infty$ and $\gamma\in\mC$ is such that $0<1/p+\Re\gamma<1$. Let
$\alpha\in SOS(\mR_+)$ and $A_{\alpha,\gamma}$ be the operator given by
\eqref{eq:operator-A}. Then
\[
A_{\alpha,\gamma}\simeq\Phi^{-1}\Op(\fg_{\omega,\gamma})\Phi,
\]
where the function $\fg_{\omega,\gamma}$, given for $(t,x)\in\mR_+\times\mR$ by
\begin{equation}\label{eq:similarity-1}
\fg_{\omega,\gamma}(t,x):=1+\frac{1}{4}\big[
e^{2\pi\Im\gamma}(1-e^{-i\omega(t)x})+
e^{-2\pi\Im\gamma}(1-e^{i\omega(t)x})
\big]r_\gamma(x) r_{\overline{\gamma}}(x)
\end{equation}
with $\omega(t):=\log[\alpha(t)/t]$, belongs to the algebra $\widetilde{\cE}(\mR_+,V(\mR))$.
\end{theorem}
\begin{proof}
From Lemmas~\ref{le:alg-A} and \ref{le:ff-fs-fr} we know that the functions
$\fr_\gamma(t,x)=r_\gamma(x)$ and $\fr_{\overline{\gamma}}(t,x)=r_{\overline{\gamma}}(x)$
with $(t,x)\in\mR_+\times\mR$ belong to $\widetilde{\cE}(\mR_+,V(\mR))$ and
\begin{equation}\label{eq:similarity-2}
R_\gamma=\Phi^{-1}\Co(r_\gamma)\Phi=\Phi^{-1}\Op(\fr_\gamma)\Phi,
\
R_{\overline{\gamma}}=\Phi^{-1}\Co(r_{\overline{\gamma}})\Phi=\Phi^{-1}\Op(\fr_{\overline{\gamma}})\Phi.
\end{equation}
From the first identity in \eqref{eq:similarity-2} and Lemma~\ref{le:shift-R-gamma}
it follows that
\begin{equation}\label{eq:similarity-3}
(I-U_\alpha^{\pm 1})R_\gamma\simeq \Phi^{-1}\Op(1-\fc_{\omega,\gamma}^\pm)\Phi,
\end{equation}
where the functions $\fc_{\omega,\gamma}^\pm$ given by \eqref{eq:shift-R-gamma} belong
to the algebra $\widetilde{\cE}(\mR_+,V(\mR))$. Then the function $\fg_{\omega,\gamma}$
given by \eqref{eq:similarity-1} belongs to the algebra $\widetilde{\cE}(\mR_+,V(\mR))$.
Combining \eqref{eq:operator-A} and \eqref{eq:similarity-1}--\eqref{eq:similarity-3}
with Theorem~\ref{th:comp-semi-commutators-PDO}, we arrive at
\begin{align*}
A_{\alpha,\gamma}
&\simeq
I + \frac{1}{4}\big[
e^{2\pi\Im\gamma}\Phi^{-1}\Op(1-\fc_{\omega,\gamma}^-)\Op(\fr_{\overline{\gamma}})\Phi
+
e^{-2\pi\Im\gamma}\Phi^{-1}\Op(1-\fc_{\omega,\gamma}^+)\Op(\fr_{\overline{\gamma}})\Phi
\big]
\\
&\simeq
I + \frac{1}{4}\big[
e^{2\pi\Im\gamma}\Phi^{-1}\Op(1-\fc_{\omega,\gamma}^-\fr_{\overline{\gamma}})\Phi
+
e^{-2\pi\Im\gamma}\Phi^{-1}\Op(1-\fc_{\omega,\gamma}^+\fr_{\overline{\gamma}})\Phi
\big]
\\
&=
\Phi^{-1}\Op(\fg_{\omega,\gamma})\Phi,
\end{align*}
which completes the proof.
\end{proof}
\begin{corollary}\label{co:similarity}
Suppose $1<p<\infty$ and $\gamma\in\mC$ is such that $0<1/p+\Re\gamma<1$ and $\alpha\in SOS(\mR_+)$.
Let $\omega(t):=\log[\alpha(t)/t]$ for $t\in\mR_+$ and the function $\fg_{\omega,\gamma}$ be
given for $(t,x)\in\mR_+\times\mR$ by \eqref{eq:similarity-1}. Then the operator $A_{\alpha,\gamma}$
given by \eqref{eq:operator-A} is Fredholm on the space $L^p(\mR_+)$ if and only if the operator
$\Op(\fg_{\omega,\gamma})$ is Fredholm on the space $L^p(\mR_+,d\mu)$ and
$\Ind A_{\alpha,\gamma}=\Ind \Op(\fg_{\omega,\gamma})$.
\end{corollary}
Thus, in order to complete the proof of Theorem~\ref{th:one-shift}, it remains to show that
the operator $\Op(\fg_{\omega,\gamma})$ is Fredholm and its index is equal to zero.
\section{Fredholm theory for the operator $\Op(\fg_{\omega,\gamma})$}
\label{sec:Fredholmness-of-my-PDO}
In this section we prove that the Mellin pseudodifferential operator with symbol
given by \eqref{eq:similarity-1} is Fredholm and its index is equal to zero. To do this,
we employ Theorem~\ref{th:Fredholmness-PDO}.
\subsection{Technical lemma}
For the calculation of the index of operator $\Op(\fg_{\omega,\gamma})$, we are going to
use a homotopic argument, linking the function $\fg_{\omega,\gamma}$ and the function that
equals identically one via a family of functions $\widetilde{\fg}_{\omega,\gamma}(\cdot,\cdot,\theta)$
with $\theta\in[0,1]$ described below.
\begin{lemma}\label{le:tech}
Let $1<p<\infty$, $\gamma\in\mC$ be such that $0<1/p+\Re\gamma<1$, and $\omega$
be a real-valued function in $SO(\mR_+)$. If
\begin{equation}\label{eq:tech-1}
\widetilde{\fg}_{\omega,\gamma}(t,x,\theta):=1+\frac{1}{4}\big[
e^{2\pi\Im\gamma}(1-e^{-i\theta\omega(t)x})+
e^{-2\pi\Im\gamma}(1-e^{i\theta\omega(t)x})
\big]r_\gamma(x) r_{\overline{\gamma}}(x)
\end{equation}
for $(t,x,\theta)\in\mR_+\times\mR\times[0,1]$,
\[
\widetilde{\fg}_{\omega,\gamma}(t,\pm\infty,\theta)
:=
\lim_{x\to\pm\infty}\widetilde{\fg}_{\omega,\gamma}(t,x,\theta)=1
\]
for $(t,\theta)\in\mR_+\times[0,1]$, and \eqref{eq:main-condition} is fulfilled,
then there is a constant $c=c(p,\gamma,\omega)>0$,
depending only on $p$, $\gamma$, and $\omega$, and such that
\begin{equation}\label{eq:tech-2}
|\widetilde{\fg}_{\omega,\gamma}(t,x,\theta)|\ge c
\quad\mbox{for all}\quad
(t,x,\theta)\in\mR_+\times\overline{\mR}\times[0,1].
\end{equation}
\end{lemma}
\begin{proof}
We follow the main lines of the proof of \cite[Lemma~5.1]{KKL14}. It is easy to see
that for every $(t,x)\in\mR_+\times\mR$,
\begin{align}
\frac{4}{r_\gamma(x)r_{\overline{\gamma}}(x)}
&=
4\sinh[\pi(x+i/p+i\gamma)]\sinh[\pi(x+i/p+i\overline{\gamma})]
\nonumber\\
&=
e^{2\pi(x+i/p+i\Re\gamma)}+e^{-2\pi(x+i/p+i\Re\gamma)}-e^{2\pi\Im\gamma}-e^{-2\pi\Im\gamma}.
\label{eq:tech-3}
\end{align}
From \eqref{eq:tech-1} and \eqref{eq:tech-3} we deduce that the function
$\widetilde{\fg}_{\omega,\gamma}$ can be rewritten as follows:
\begin{align*}
&\widetilde{\fg}_{\omega,\gamma}(t,x,\theta)
=
\big[
e^{2\pi(x+i/p+i\Re\gamma)}+e^{-2\pi(x+i/p+i\Re\gamma)}-e^{2\pi\Im\gamma}-e^{-2\pi\Im\gamma}
\\
&\quad+
e^{2\pi\Im\gamma}(1-e^{-i\theta\omega(t)x})+e^{-2\pi\Im\gamma}(1-e^{i\theta\omega(t)x})
\big]\frac{r_\gamma(x) r_{\overline{\gamma}}(x)}{4}
\\
&=
\Big[
\frac{e^{2\pi(x+i/p+i\Re\gamma)}+e^{-2\pi(x+i/p+i\Re\gamma)}}{2}-
\frac{e^{i\theta\omega(t)x-2\pi\Im\gamma}
+
e^{2\pi\Im\gamma-i\theta\omega(t)x}}{2}
\Big]\frac{r_\gamma(x) r_{\overline{\gamma}}(x)}{2}
\\
&=
\frac{\cosh[2\pi(x+i/p+i\Re\gamma)]-\cosh[i\theta\omega(t)x-2\pi\Im\gamma]}{2}r_\gamma(x) r_{\overline{\gamma}}(x)
\\
&=
\frac{\sinh[\pi(x-\Im\gamma+i/p+i\Re\gamma)+i\theta\omega(t)x/2]}{\sinh[\pi(x-\Im\gamma+i/p+i\Re\gamma)]}
\\
&\quad\times
\frac{\sinh[\pi(x+\Im\gamma+i/p+i\Re\gamma)-i\theta\omega(t)x/2]}{\sinh[\pi(x+\Im\gamma+i/p+i\Re\gamma)]},
\end{align*}
whence, by \cite[formula 4.5.54]{AS72},
\begin{align}
|\widetilde{\fg}_{\omega,\gamma}(t,x,\theta)|
=&
\sqrt{\frac{\sinh^2[\pi(x-\Im\gamma)]+\sin^2[\pi(1/p+\Re\gamma)+\theta\omega(t)x/2]}{\sinh^2[\pi(x-\Im\gamma)]+\sin^2[\pi(1/p+\Re\gamma)]}}
\nonumber\\
&\times
\sqrt{\frac{\sinh^2[\pi(x+\Im\gamma)]+\sin^2[\pi(1/p+\Re\gamma)-\theta\omega(t)x/2]}{\sinh^2[\pi(x+\Im\gamma)]+\sin^2[\pi(1/p+\Re\gamma)]}}
\nonumber\\
&
=:\fg_+(t,x,\theta)\fg_-(t,x,\theta).
\label{eq:tech-4}
\end{align}
Thus, we need to estimate from below the functions
\[
\fg_\pm(t,x\pm\Im\gamma,\theta)=
\sqrt{\frac{\sinh^2(\pi x)+\sin^2[\pi(1/p+\Re\gamma)
+\theta\omega(t)\Im\gamma/2\pm\theta\omega(t)x/2]}
{\sinh^2(\pi x)+\sin^2[\pi(1/p+\Re\gamma)]}}.
\]
From $0<1/p+\Re\gamma<1$ and \eqref{eq:main-condition} it follows that
\begin{align*}
\mathcal{I}
&:=
\frac{1}{p}+\Re\gamma+\min\left(0,\frac{1}{2\pi}\inf_{t\in\mR_+}(\omega(t)\Im\gamma)\right)>0,
\\
\mathcal{S}
&:=
\frac{1}{p}+\Re\gamma+\max\left(0,\frac{1}{2\pi}\sup_{t\in\mR_+}(\omega(t)\Im\gamma)\right)<1.
\end{align*}
It is easy to see that
\begin{equation}\label{eq:tech-5}
\mathcal{I}\le\frac{1}{p}+\Re\gamma+\frac{\theta\omega(t)\Im\gamma}{2\pi}\le\mathcal{S},
\quad t\in\mR_+,\quad \theta\in[0,1].
\end{equation}
Let
\[
M(\omega):=\sup_{t\in\mR_+}|\omega(t)|
\]
and $q\in[2,+\infty)$ be defined by $1/q:=\min(\mathcal{I},1-\mathcal{S})$. Then
\begin{equation}\label{eq:tech-6}
\left|\frac{\theta\omega(t)x}{2}\right|
\le\frac{M(\omega)}{2}|x|\le\frac{\pi}{2q}\quad\mbox{if}\quad
|x|\le\frac{\pi}{qM(\omega)},
\quad t\in\mR_+,
\quad \theta\in[0,1].
\end{equation}
Further, from \eqref{eq:tech-5}--\eqref{eq:tech-6} we obtain
\begin{align*}
\frac{\pi}{2q}
&=\pi\frac{1}{q}-\frac{\pi}{2q}
\le\pi\mathcal{I}-\frac{\pi}{2q}
\le
\pi\left(\frac{1}{p}+\Re\gamma\right)+\frac{\theta\omega(t)\Im\gamma}{2}
\pm\frac{\theta\omega(t)x}{2}
\\
&\le\pi\mathcal{S}+\frac{\pi}{2q}\le
\pi\left(1-\frac{1}{q}\right)+\frac{\pi}{2q}
=
\pi-\frac{\pi}{2q}
\end{align*}
for all $|x|\le\frac{\pi}{qM(\omega)}$, $t\in\mR_+$, and $\theta\in[0,1]$. Put
\[
J_\pm:=\left[-\frac{\pi}{qM(\omega)}\pm\Im\gamma,\frac{\pi}{qM(\omega)}\pm\Im\gamma\right].
\]
Hence
\[
\frac{\pi}{2q}
\le
\pi\left(\frac{1}{\pi}+\Re\gamma\right)\pm\frac{\theta\omega(t)x}{2}
\le
\pi-\frac{\pi}{2q}\quad\mbox{if}
\quad x\in J_\pm,
\quad t\in\mR_+,
\quad \theta\in[0,1],
\]
whence
\begin{equation}\label{eq:tech-7}
\sin^2\left[
\pi\left(\frac{1}{\pi}+\Re\gamma\right)\pm\frac{\theta\omega(t)x}{2}
\right]
\ge
\sin^2\left(\frac{\pi}{2q}\right),
\ x\in J_\pm,
\ t\in\mR_+,
\ \theta\in[0,1].
\end{equation}
Since the functions $\varphi_\pm(x)=\sinh^2[\pi(x\pm\Im\gamma)]$ are
increasing on $[\mp\Im\gamma,+\infty)$ and are decreasing on
$(-\infty,\mp\Im\gamma]$, taking into account \eqref{eq:tech-7},
we obtain
\begin{equation}\label{eq:tech-8}
\fg_\pm(t,x,\theta)
\ge
\frac{\sin[\pi/(2q)]}
{\sqrt{\sinh^2[\pi^2/(qM(\omega))]+\sin^2[\pi(1/p+\Re\gamma)]}}
=:c_1(p,\gamma,\omega)>0
\end{equation}
for all $x\in J_\pm$, $t\in\mR_+$, and $\theta\in[0,1]$. On the other hand,
\begin{align}
\fg_\pm(t,x,\theta)
&\ge
\sqrt{\frac{\sinh^2[\pi(x\mp\Im\gamma)]}{\sinh^2[\pi(x\mp\Im\gamma)]+\sin^2[\pi(1/p+\Re\gamma)]}}
\nonumber\\
&=
\left(1+\sin^2[\pi(1/p+\Re\gamma)]\sinh^{-2}[\pi(x\mp\Im\gamma)]\right)^{-1/2}
\nonumber\\
&\ge
\left(1+\sin^2[\pi(1/p+\Re\gamma)]\sinh^{-2}[\pi^2/(qM(\omega))]\right)^{-1/2}
\nonumber\\
&=:
c_2(p,\gamma,\omega)>0
\label{eq:tech-9}
\end{align}
for all $x\in\mR\setminus J_\pm$, $t\in\mR_+$, and $\theta\in[0,1]$.
Combining \eqref{eq:tech-4} and \eqref{eq:tech-8}--\eqref{eq:tech-9},
we arrive at \eqref{eq:tech-2} with
\[
c:=c(p,\gamma,\omega)=\big(\min(c_1(p,\gamma,\omega),c_2(p,\gamma,\omega))\big)^2,
\]
which completes the proof.
\end{proof}
\subsection{Fredholmness and index of the operator $\Op(\fg_{\omega,\gamma})$}
Now we are ready to finish the proof of Theorem~\ref{th:one-shift}.
\begin{theorem}\label{th:Fredholmness-of-my-PDO}
Suppose $1<p<\infty$ and $\gamma\in\mC$ is such that $0<1/p+\Re\gamma<1$. Let
$\omega\in SO(\mR_+)$ be a real-valued function and the function $\fg_{\omega,\gamma}$
be given for $(t,x)\in\mR_+\times\mR$ by \eqref{eq:similarity-1}. If \eqref{eq:main-condition}
holds, then the operator $\Op(\fg_{\omega,\gamma})$ is Fredholm on the space $L^p(\mR_+,d\mu)$
and $\Ind\Op(\fg_{\omega,\gamma})=0$.
\end{theorem}
\begin{proof}
The proof is developed exactly as in \cite[Theorem~5.2]{KKL14}.
In view of Lemmas~\ref{le:ff-fs-fr} and~\ref{le:fb}, the function
$\fg_{\omega,\gamma}$ belongs to $\widetilde{\cE}(\mR_+,V(\mR))$.
From \eqref{eq:similarity-1} and
Lemmas~\ref{le:values} and \ref{le:tech} it follows that
\[
\fg_{\omega,\gamma}(t,\pm\infty)=1\ne 0
\quad\mbox{for all}\quad t\in\mR_+,
\quad
\fg_{\omega,\gamma}(\xi,x)\ne 0
\quad\mbox{for all}\quad
(\xi,x)\in\Delta\times\overline{\mR}.
\]
Thus, by Theorem~\ref{th:Fredholmness-PDO}, the operator $\Op(\fg_{\omega,\gamma})$
is Fredholm on $L^p(\mR_+,d\mu)$.

For $\tau>1$, consider $\Pi_\tau:=[\tau^{-1},\tau]\times\overline{\mR}$.
Since the function $\widetilde{\fg}_{\omega,\gamma}$ given by \eqref{eq:tech-1}
is continuous and separated from $0$ for all
$(t,x,\theta)\in\mR_+\times\overline{\mR}\times[0,1]$
in view of Lemma~\ref{le:tech}, we conclude that
$\{\arg\widetilde{\fg}_{\omega,\gamma}(t,x,\theta)\}_{(t,x)\in\partial\Pi_\tau}$
does not depend on $\theta\in[0,1]$. Consequently,
\begin{equation}\label{eq:Fredholmness-of-my-PDO-1}
\{\arg\fg_{\omega,\gamma}(t,x)\}_{(t,x)\in\partial\Pi_\tau}
=
\{\arg\widetilde{\fg}_{\omega,\gamma}(t,x,0)\}_{(t,x)\in\partial\Pi_\tau}=0.
\end{equation}
By Theorem~\ref{th:Fredholmness-PDO} and \eqref{eq:Fredholmness-of-my-PDO-1},
\[
\Ind\operatorname{Op}(\fg_{\omega,\gamma})
=
\lim_{\tau\to\infty}\frac{1}{2\pi}\{\arg\fg_{\omega,\gamma}(t,x)\}_{(t,x)\in\partial\Pi_\tau}
=0,
\]
which completes the proof of the theorem.
\end{proof}
Finally, Theorem~\ref{th:one-shift} follows from Theorem~\ref{th:Fredholmness-of-my-PDO},
Corollary~\ref{co:similarity}, and Theorem~\ref{th:reduction}.

\bigskip
\noindent
{\bf Acknowledgement.}
This work was partially supported by the Funda\c{c}\~ao para a Ci\^encia e a Tecnologia
(Portuguese Foundation for Science and Technology) through the project
PEst-OE/MAT/UI0297/2014 (Centro de Matem\'atica e Apli\-ca\c{c}\~oes).
The authors would like to thank the anonymous referee for useful remarks
and for informing about the work \cite{DKT13}.
\bibliographystyle{amsplain}

\end{document}